\documentclass{compositio} 

\usepackage{amsmath,amssymb,amsthm} 
\usepackage[all]{xy}
\usepackage{extarrows}
\usepackage{graphicx}
\usepackage{bm}
\usepackage{comment, url}
\usepackage[normalem]{ulem}

\usepackage[pagewise]{lineno}
\usepackage{wasysym}

\usepackage{time}

\newtheorem{thm}{Theorem}[section]

\newtheorem{cor}[thm]{Corollary}
\newtheorem{prop}[thm]{Proposition}  

\theoremstyle{remark}

\theoremstyle{definition}

\newtheorem{def/prop}[thm]{Definition/Proposition}

\numberwithin{equation}{section}

\renewcommand\proofname{\bf Proof}

\usepackage[usenames]{color}
\newcommand{\rs}[1]{} \newcommand{\rma}[1]{}

\def\Im{\mathop{\mathrm{Im}}\nolimits}

\def\Gal{\mathop{\mathrm{Gal}}\nolimits}
\def\Spec{\mathop{\mathrm{Spec}}\nolimits}

\def\Cl{\mathop{\mathrm{Cl}}\nolimits}

\newcommand{\mf}[1]{{\mathfrak{#1}}}

\newcommand{\bb}[1]{{\mathbb{#1}}}
\newcommand{\mca}[1]{{\mathcal{#1}}}

\newcommand{\To}{\longrightarrow}

\newcommand{\inj}{\hookrightarrow}
\newcommand{\surj}{\twoheadrightarrow}

\newcommand{\congto}{\overset{\cong}{\to}}

\newcommand{\imp}{\Longrightarrow}

\newcommand{\Z}{\bb{Z}}

\newcommand{\Q}{\bb{Q}}

\newcommand{\C}{\bb{C}}
\newcommand{\F}{\bb{F}}
\newcommand{\p}{\mf{p}}

\newcommand{\ol}{\overline}

\newcommand{\ds}{\displaystyle}

\newcommand{\wt}[1]{{\widetilde{#1}}}
\newcommand{\wh}[1]{{\widehat{#1}}}

\DeclareMathOperator*{\restprod}%
 {\mathchoice{\ooalign{\ensuremath{\displaystyle\prod}\crcr\ensuremath{\displaystyle\coprod}}}%
             {\ooalign{\ensuremath{\textstyle\prod}\crcr\ensuremath{\textstyle\coprod}}}%
             {\ooalign{\ensuremath{\scriptstyle\prod}\crcr\ensuremath{\scriptstyle\coprod}}}%
             {\ooalign{\ensuremath{\scriptscriptstyle\prod}\crcr\ensuremath{\scriptscriptstyle\coprod}}}%
 }

\newcommand{\pmx}[1]{\begin{pmatrix}#1\end{pmatrix}}
\newcommand{\spmx}[1]{{\small \pmx{#1}}}

\title[A Hilbert reciprocity law on 3-manifolds]
{\ \\[-1cm] A Hilbert reciprocity law on 3-manifolds}

\author{Hirofumi Niibo}
\email{niibo.hirofumi@gmail.com}
\address{Supership Inc.; Toranomon-hills 27F, 1-17-1 Toranomon, Minato-ku, 105-6427, Tokyo, Japan} 

\author{Jun Ueki}
\email{uekijun46@gmail.com}
\address{Department of Mathematics, Faculty of Science, Ochanomizu University\\ 
2-1-1 Otsuka, Bunkyo-ku, 112-8610, Tokyo, Japan}



\begin{document}

\begin{abstract} 
Based on our homological idelic class field theory, we formulate an analogue of the Hilbert reciprocity law on a rational homology 3-sphere endowed with an infinite link, in the spirit of arithmetic topology; 
We regard the intersection form on the unitary normal bundle of each knot as an analogue of the Hilbert symbol at each prime ideal to formulate the Hilbert reciprocity law, ensuring that cyclic covers of links are analogues of Kummer extensions. 
\end{abstract}

\maketitle 

{\small 
\tableofcontents 
}

\section{Introduction} 
In this article, we formulate an analogue of \emph{the Hilbert reciprocity law} 
in a view of homological idelic class field theory for 3-manifolds \cite{Niibo1, NiiboUeki}, 
that may be compatible or comparable with Mihara's cohomological one \cite{Mihara2019Canada} and Morishita et.al.'s one for foliated dynamical systems \cite{KimMorishitaNodaTerashima2021}.

We mainly work on a rational homology 3-sphere $M$ ($\Q$HS$^3$) 
endowed with a link $\mca{K}$ with countably infinite tame components. 
We mostly only assume that components of $\mca{K}$ generate $H_1(M)=H_1(M,\Z)$, nevertheless 
we verified in \cite{Ueki7} the following implications amongst nice conditions on $\mca{K}$; 
Chebotarev \cite{Mazur2012, McMullen2013CM} $\imp$ stably generic \cite{Mihara2019Canada} $\imp$ very admissible \cite{Niibo1, NiiboUeki}. 
By virtue of McMullen's result, we have various such $\mca{K}$'s with their own interest \cite{Ueki9}. 

We follow Neukirch's description in \cite[Chapter VI, Corollary 5.7, Theorem 8.1]{Neukirch}. 
Namely, we first formulate the product formula of norm residue symbols. 
Secondary, regarding the intersection form on each torus as an analogue of the Hilbert symbol, we prove an analogue of the Hilbert reciprocity law.
Finally, we verify that in our context the analogues of Kummer extensions are nothing but cyclic covers of links, ensuring Hirano's argument on arithmetic Dijkgraaf--Witten invariants in \cite[Appendix]{Hirano1-arXiv}. 

\section{M$^2$KR dictionary} 
Let us first exhibit the dictionary of basic analogies between prime numbers and knots (cf.\cite{Morishita2012}); 
{
\begin{center}
\begin{tabular}{|c||c|}
\hline 
Number theory&Low dimensional topology\\ 
\hline \hline
number field $k$ (the ring of integers $\Spec\mca{O}_k$) & connected oriented closed 3-manifold $M$\\
prime ideal $\p: \Spec \F_\p\inj \Spec \mca{O}_k$ & knot $K: S^1\inj M$\\
family of prime ideals $S=\{\p_1,...,\p_s\}$& link $L:
\sqcup S^1\inj M$\\
(ramified/unramified) extension $l/k$& (branched/unbranched) cover $h:N\to M$\\
%
\hline 
\'etale fundamental group $\pi_1^{\text{\'et}}(\Spec \mca{O}_k)$ & fundamental group $\pi_1(M)$\\
$\pi_1^{\text{\'et}}(\Spec \mca{O}_k-S)$  & $\pi_1(M-L)$\\
geometric point $x:\Spec \C\to \Spec \mca{O}_k$ & base point $b_m:\{{\rm pt}\} \inj M$\\ 
\hline 
ideal group $I_k$ & 1-cycle group $Z_1(M)$\\
$\partial: k^\times\to I_k; a\mapsto (a)$ & $\partial: C_2(M)\to Z_1(M); s\mapsto \partial s$\\
principal ideal group $P_k={\rm Im}\partial$ & 1-boundary group $B_1(M)={\rm Im}\partial$ \\ 
\hline 
ideal class group $\Cl(k)=I_k/P_k$ & 1st homology group $H_1(M)=Z_1(M)/B_1(M)$\\
(Fact: $\#\Cl(k)<\infty$)& (Assumption: $\#H_1(M)<\infty$)\\ 
\hline 
Artin reciprocity law: & Hurewicz isomorphism:\\
$\pi_1^{\text{\'et}}(\ol{\Spec \mca{O}_k})^{\rm ab} \cong \Gal(k^{\rm ur}_{\rm ab}/k) \cong \Cl(k) $ &$\pi_1(M)^{\rm ab} \cong \Gal(M_{\rm ab}/M) \cong H_1(M)$ \\
\hline 
Hilbert theory & Hilbert theory  \\ 
\hline 
\end{tabular} 
\end{center} 
Here, we put $\ol{\Spec \mca{O}_k}=\Spec \mca{O}_k\cup\{\text{infinite primes}\}$. 
The maximal abelian cover and the unramified abelian extension are denoted by $M^{\rm ab}\to M$ and $k^{\rm ab}_{\rm ur}/k$ respectively. 
The Hilbert theory means that on decompositions of prime ideals in Galois extensions and its analogue (cf.\cite[Section 5]{Morishita2012}, \cite[Section 2]{Ueki1}. 

Next, let us recall the local theories. Let $\p$ be a prime ideal of a number field $k$ and let $k_\p$, $\mca{O}_\p$, and $\F_\p$ denote the local field, local integer ring, and the residual field of $\p$ respectively. Let $V_K$ be a tubular neighborhood of a knot $K$ in a 3-manifold. 
Let $\simeq$ denote the (\'etale) homotopy equivalence. Then, we have the following. 

{
\begin{center}
\begin{tabular}{|c||c|} 
\hline 
$\Spec \mca{O}_\p \simeq \Spec \F_\p$ & $V_K\simeq K$\\
\hline 
$\Spec k_\p \simeq \Spec \mca{O}_\p- \Spec \F_\p$ & $\partial V_K\simeq V_K-K$\\ 
\hline 
$1\to \mca{O}_\p^{\times}\to k_\p^{\times}\overset{v_\p}{\to} \Z\to 0$ & $0\to H_2(V_K,\partial V_K) \overset{\partial_*}{\to} H_1(\partial V_K)\overset{v_K}{\to} \Z \to 0$ \\ 
\hline 
\end{tabular} 
\end{center}
}


Finally, we recollect the homological idelic class field theory due to the authors \cite{NiiboUeki}. 
Let $M$ be a connected, oriented, closed 3-manifold endowed with a link $\mca{K}$ consisting of countably infinite tame components and suppose that $\mca{K}$ is \emph{admissible}, that is, its components generate the group $H_1(M)$. 

To each knot $K$ contained in $\mca{K}$, we associate a torus by \emph{blow up}, that is, replacing $K$ by its unitary normal bundle $N^1_M(K)$. 
The result is naturally homeomorphic to the exterior of a tubular neighborhood $V_K$ of $K$. 
In what follows. we work under the identification $N^1_M(K)=\partial V_K$. 
In this manner, we may avoid with less effort the intersection of tubular neighborhoods of knots. 

The idele group of the pair $(M,\mca{K})$ is defined by 
\[\displaystyle \mca{J}_{M,\mca{K}}:=
\restprod_{K\subset \mca{K}}H_1(\partial V_K)=
\bigl\{(a_K)_K\in \prod_{K\subset \mca{K}}H_1(\partial V_K) \mid 
v_{K}(a_K)=0 {\rm \ for \ all \ but \ finite \ number \ of \ }K
\bigr\}\]
as a group, where $K$ runs through components of $\mca{K}$. 
For each $K$, define the local norm topology on $\mu_K={\rm Ker} v_K$ so that the set of subgroups of finite indices is a neighborhood basis of the identity and that on $H_1(\partial V_K)$ so that $H_1(\partial V_K)$ is a topological group and $\mu_K \inj H_1(\partial V_K)$ is an open and continuous map. Then $\mca{J}_{M,\mca{K}}$ is a topological group which is the restricted product with respect to the family $\{\mu_K\}_{K\subset \mca{K}}$ of open subgroups. 

Suppose that $L$ runs through finite sublinks of $\mca{K}$. Then we have natural isomorphisms $\ds H_1(M-\mca{K})\cong \varprojlim_{L\subset \mca{K}} H_1(M-L)$ and $H_2(M,\mca{K})$ $\ds \cong \varinjlim_{L\subset \mca{K}}H_2(M,L)$ 
yielding a natural surjective homomorphism $\wt{\rho}:\mca{J}_{M,\mca{K}}\to H_1(M-\mca{K})$ and a homomorphism $\Delta:H_2(M,\mca{K})\to \mca{J}_{M,\mca{K}}$ satisfying the key equality ${\rm Ker}\wt{\rho}={\rm Im}\Delta$. 
This $\wt{\rho}$ is compatible with the maps of local reciprocities and induces $\rho: \mca{C}_{M,\mca{K}}\to H_1(M-\mca{K})$. 

Now suppose in addition that $\mca{K}$ is \emph{very admissible}, that is, for any finite branched cover $h:N\to M$ branched along a finite sublink $L$ of $\mca{K}$, the group $H_1(N-h^{-1}(L))$ is generated by $h^{-1}(\mca{K})$. 
Then for each $h$, a natural isomorphism $\rho_h:\mca{C}_{M,\mca{K}}/h_*(\mca{C}_{N,h^{-1}(\mca{K})})\congto {\rm Gal}(h)^{\rm ab}$ is induced. This fact may be seen as analogue of the Artin reciprocity law. 

{
\begin{center} \begin{tabular}{|c||c|}
\hline 
idele group $\mca{J}_k$ & 
$\mca{J}_{M,\mca{K}}$\\ \hline 
diagonal embedding $\Delta:k^\times\to \mca{J}_k$ & $\Delta:H_2(M,\mca{K})\to \mca{J}_{M,\mca{K}}$ \\ \hline 
principal idele group $\mca{P}_k:=\Im \Delta$ & 
$\mca{P}_{M,\mca{K}}:=\Im \Delta$ \\ \hline 
idele class group $\mca{C}_k:=\mca{J}_k/\mca{P}_k$ & 
$\mca{C}_{M,\mca{K}}:=\mca{J}_{M,\mca{K}}/\mca{P}_{M,\mca{K}}$\\ 
\hline 
Artin reciprocity law & the global reciprocity law \\ \hline
\end{tabular} 
\end{center} 
}

A key to the global reciprocity law was the equality ${\rm Im \Delta}={\rm Ker}(\wt{\rho}:\mca{J}_{M,\mca{K}}\surj \mca{C}_{M,\mca{K}})$ asserted in \cite[Theorem 5.3]{NiiboUeki}. 
We may notice that at the stage to prove this equality, we only used that $\mca{K}$ is admissible, 
but not that $\mca{K}$ is very admissible.

\section{Number theory}
We recollect the Hilbert reciprocity law using of norm residue symbols, whose analogue will be discussed later. 
\subsection{The norm residue symbols} 

Let $l/k$ be a finite abelian extension of a number field. 
Then the inverse map of the Artin reciprocity map $\Gal(l/k)\congto \mca{C}_k/{\rm Nr}_{l/k} \mca{C}_l$ induces 
a surjective homomorphism $(\ , l/k):\mca{C}_k\surj \Gal(l/k)$ called the global norm residue symbol. 

Let $\p$ be a non-zero prime ideal of $\mca{O}_k$ and put $l_\p=k_\p l$. Then the inverse map of the local reciprocity map $\Gal(l_\p/k_\p)\congto k_\p^*/{\rm Nr}_{l_\p/k_\p} l_\p^*$ induces a surjective homomorphism $(\ ,l_\p/k_\p):k_\p^*\surj \Gal(l_\p/k_\p)$ called the local norm residue symbol.

\begin{prop} {\rm (cf.\cite[Chapter VI, Proposition 5.6]{Neukirch})} \label{prop.normresidue} 
Suppose that $l/k$ is a finite abelian extension.
The norm residue symbols, 
a natural injective homomorphism $\langle\,\rangle_\p: k_\p^*\to \mca{C}_k$, 
and the natural embedding $G(l_p/k_p) \to G(l/k)$ commute{\rm ;}  
\[ \xymatrix{
k_p^*\ \ar@{^{(}->}[d]_{\langle \ \rangle_\p} \ar@{->>}[r]^-{(\ ,\, l_p/k_p)} \ar@{}[dr]|\circlearrowleft & \ \Gal(l_p/k_p) \ar[d] \\
\mca{C}_k\ \ar@{->>}[r]_-{(\ ,\, l/k)} & \ \Gal(l/k).
} \]  

\end{prop} 

\begin{prop} {\rm (cf.\cite[Chapter VI, Corollary 5.7]{Neukirch})} \label{prop.product} 
Let $l/k$ be a finite abelian extension and let $\alpha=(\alpha_\p)_\p \in \mca{J}_k$ be an element of the idele group. 
For each prime ideal $\p$ of $k$, put $l_\p=k_\p l$. 
Then the global/local norm residue symbols satisfy  
\[(\alpha, l/k)=\prod_\p (\alpha_\p, l_\p/k_\p).\]
If in addition $\alpha =\Delta(a)$ for some $a\in k^*$, that is, if $\alpha \in \mca{P}_k$, then $(\alpha, l/k)=1$. 
\end{prop} 

We remark that the assumption ``abelian'' on $l/k$ and the groups ${\rm Gal}(l/k)$ and ${\rm Gal}(l_\p/k_\p)$ in above may be replaced by ``Galois'', ${\rm Gal}(l/k)^{\rm ab}$, and ${\rm Gal}(l_\p/k_\p)^{\rm ab}$ by a slight additional argument. 

\subsection{The Hilbert reciprocity} \label{sec.Hilbert} 
Let $n \in \Z_{>1}$ and let $k$ be a number field containing primitive $n$-th roots of unity. 
Let $\mu_n$ denote the set of $n$-th roots of unity in $k$. 
Then for each $b \in k^\ast :=\{x\in k\mid x\neq 0\}$, 
the extension $k(\sqrt[n]b)/k$ obtained by adding a primitive $n$-th root $\sqrt[n]b$ of $b$ is independent of the choice of $\sqrt[n]b$, and is called the Kummer extension. Here, $k$ may be replaced by a local field $k_\p$, $\p$ being a non-zero prime ideal of $\mca{O}_k$. 

The $n$-th Hilbert symbol 
\[ \left( \dfrac{\bullet,\bullet}{\p} \right) :k_\p^*\times k_\p^*\to \mu_n \cong \Z/n\Z \]
is defined by using the norm residue symbol as 
\[(a,k_\p(\sqrt[n]{b})/k_\p)\sqrt[n]{b}=\left( \dfrac{a,b}{\p}\right)\sqrt[n]{b}.\]
The product formula of the norm residue symbol yields the following. 

\begin{prop}[{\cite[Chapter VI, Theorem 8.1]{Neukirch}}] \label{prop.Hilbert}
For each $a,b \in k^*$, \[\prod_\p \left( \dfrac{a,b}{\p}\right) =1.\]
\end{prop}


\section{Meridians and longitudes}
We define the meridian and longitude of a knot in a $\Q$HS$^3$ $M$. 

Let $K$ be a knot in a $\Q$HS$^3$ $M$. 
Let $[\Sigma_K]$ be a generator of $H_2(M-{\rm Int}V_K,\partial V_K)$ whose image under $H_2(M-{\rm Int}V_K,\partial V_K)\overset{\partial}{\To}H_1(\partial V_K)\surj H_1(V_K)$ coincides with $[K]$. 
The image of $\partial [\Sigma_K]$ in $H_1(\partial V_K)$ is called the preferred longitude of $K$ and is denoted by $\lambda=\lambda_K$. 

A generator of ${\rm Ker}(H_1(\partial V_K)\surj H_1(V_K))$ which is clockwise to $K$, supposing that $K$ sticks into the wall,
is called a meridian and is denoted by $\mu=\mu_K$.

\section{Intersection number}
Let us briefly recall the Poincar\'e--Lefschetz theory on intersections of homology cycles on manifolds (cf.\cite{Spanier1966book}). 
If $M$ is an $n$-dimensional manifold, then the intersection number $\iota_M=\cap_M: H_k(M,\Z)\times H_{n-k}(M,\Z)\to \Z$ is well-defined on the chain level by the number of intersection points of cycles with signs. 
This may be translated to the cohomology side via the Poincare duality; Define $\iota_M^*:H^k(M;\Z)\times H^{n-k}(M,\Z)\to \Z; (a,b)\mapsto \langle a\cup b,[M]\rangle$ by the cup product.
Let $(x^*,y^*)\in H^k(M;\Z)\times H^{n-k}(M,\Z)$ denote the dual of $(x,y)\in H_k(M,\Z)\times H_{n-k}(M,\Z)$. 
Then we have $\iota_M^*(x^*,y^*)=\iota_M(x,y)$. 
The form $\iota_M$ induces that on $H_*/{\rm (torsion)}$'s. 
We have skew symmetricity $\iota_M(x,y)=(-1)^{k(n-k)}(\iota_M(x,y))$. 

\section{Norm residue symbols} 
In this section, we point out analogues of Propositions \ref{prop.normresidue} and \ref{prop.product}. 

Let $M$ be a $\Q$HS$^3$ endowed with an admissible link $\mca{K}$. 
For a finite abelian cover $h:N\to M$ branched along a finite sublink of $\mca{K}$, 
we define \emph{the norm residue symbols} 
$(\ ,h):\mca{C}_{M,\mca{K}}\to {\rm Gal}(h)$
to be the composition of the quotient map and the global reciprocity map. 

For each knot $K\subset \mca{K}$, let $ \langle \ \rangle_K: H_1(\partial V_K)\inj \mca{J}_{M,\mca{K}}$ denote the natural injective homomorphism and let $ \langle \ \rangle_K: H_1(\partial V_K)\inj \mca{J}_{M,\mca{K}}\surj \mca{C}_{M,\mca{K}}$ also denote the composition with slight abuse of notation.  
Let $h_K=h|_{h^{-1}(\partial V_K)}$ denote the restriction map. Note that there is a natural map ${\rm Gal}(h_K)={\rm Gal}(h_K)\to {\rm Gal}(h)$. Define the local norm residue symbol in a similar way as $(\ , h_K):H_1(\partial V_K)\to {\rm Gal}(h_K)$. Then the compatibility of the local and global norm residue maps is stated as follows. 
\begin{prop} 
The following diagram commutes. 
\[ \xymatrix{
H_1(\partial V_K)\ar@{^{(}->}[d]_{\langle \ \rangle_K} \ar@{->>}[r]^{ (\ ,\, h_K)} \ar@{}[dr]|\circlearrowleft & \Gal(h_K) \ar[d] \\
\mca{C}_{M,\mca{K}} \ar@{->>}[r]_{(\ ,\, h)} & \Gal(h)}.
\]  
\end{prop} 

The idele group $J_{M,\mca{K}}$ is topologically generated by the set $\bigcup_K {\rm Im}(\langle \ \rangle_K: H_1(\partial V_K)\inj \mca{J}_{M,\mca{K}})$, that is, the set of elements presented as $\alpha=\langle a \rangle_K$ for some $K$ and $a\in H_1(\partial V_K)$. 
If we fix a knot $K'\subset \mca{K}$, then each topological generator $\alpha=(\alpha_K)_K=\langle a \rangle_{K'}$ with $a\in H_1(\partial V_{K'})$ satisfies 
$\alpha_K=
\left\{ \begin{array}{ll} 
0 & {\rm if} \ K\neq K'\\
a & {\rm if} \ K=K' 
 \end{array}\right. $ 
 and hence 
 $(\alpha,h)=(\langle a\rangle_{K'}, h)=(a, h_{K'}) =\prod_{K\subset \mca{K}} (\alpha_K, h_{K'}).$
 This equality yields the following. 
\begin{thm}[(The product formula)] \label{thm.product} 
Let $M$ be a $\Q$HS$^3$ endowed with an admissible link $\mca{K}$. 
Let $h:N\to M$ be a finite abelian cover branched along a finite sublink of $\mca{K}$ and 
let $\alpha\in \mca{J}_{M,\mca{K}}$. Then \[(\alpha, h)=\prod_{K\subset \mca{K}} (\alpha_K, h_K)\] holds. 
If in addition $\alpha\in \mca{P}_{M,\mca{K}}$, then $(\alpha, h)={\rm id}$ holds. 
\end{thm} 

\begin{proof} We prove the latter half of the assertion. 
Recall that the group $\mca{P}_{M,\mca{K}}$ is defined as the image of f the natural map 
$\ds \Delta=\prod_K {\rm pr}_K\circ \partial:H_2(M,\mca{K})=\varinjlim_{L\subset \mca{K}} H_2(M,L) \to \mca{J}_{M,\mca{K}}$
and the equality $\mca{P}_{M,\mca{K}}={\rm Ker}(\mca{J}_{M,\mca{K}}\surj \mca{C}_{M,\mca{K}})$ holds. 
Hence if $\alpha \in \mca{P}_{M,\mca{K}}$, then by $\alpha\in {\rm Ker}(\mca{J}_{M,\mca{K}}\to \mca{C}_{M,\mca{K}}\surj \Gal(h))$, and hence $(\alpha,h)={\rm id}$. 
\end{proof} 

We remark that the assumption ``abelian'' on $h:N\to M$ and the group ${\rm Gal}(h)$ in above may be replaced by 
``Galois'' and ${\rm Gal}(h)^{\rm ab}$. 
Even in this case, we have ${\rm Gal}(h_K)={\rm Gal}(h_K)^{\rm ab}$. 

\section{The Hilbert reciprocity} 
In this section, we regard the intersection number on a torus as an analogue of the Hilbert symbol
to formulate an analogue of the Hilbert reciprocity law and attach remarks. 

Let $M$ be a $\Q$HS$^3$ endowed with a link $\mca{K}$ with countably infinite tame components. 
For each knot $K \subset \mca{K}$, 
the pair $(\mu_K,\lambda_K)$ of the meridian and the longitude forms a basis of $H_1(\partial V_K)$, 
so each element may be written as $x\mu_K+y\lambda_K=\spmx{x\\y} \in H_1(\partial V_K)$, $x,y \in \Z$. 
The intersection number of $\spmx{x_1\\ y_1}$, $\spmx{x_2\\ y_2}$ $\in H_1(\partial V_K)$ is given by 
\[ \iota_K(\spmx{x_1\\y_1},\spmx{x_2\\ y_2})={\rm det}\spmx{x_1&x_2\\y_1&y_2}=x_1y_2-x_2y_1.\]
If representing cycles $c_1$ and $c_2$ of $\spmx{x_1\\ y_1}$ and $\spmx{x_2\\ y_2}$ intersect transversely, 
then $\iota_K(\spmx{x_1\\y_1},\spmx{x_2\\ y_2})=\iota_K(c_1,c_2)$ holds. 

\begin{thm}[(The Hilbert reciprocity law)] \label{thm} 
Let $M$ be a $\Q$HS$^3$ endowed with a link $\mca{K}$ with countably infinite tame components and suppose that $K$ runs through knots in $\mca{K}$. 

{\rm i)} If $a,b \in \mca{J}_{M,\mca{K}}$, then $\iota(a,b):=\sum_K \iota_K ({\rm pr}_K(a), {\rm pr}_K(b))$ is a well-defined finite sum.

{\rm ii)} If $a,b \in \mca{P}_{M,\mca{K}}$, then $\iota(a,b)=0$ holds. 

\end{thm} 

The assertion ii) may be seen as an analogue of the Hilbert reciprocity law (Proposition \ref{prop.Hilbert}). 
In the proof, we use the product formula of norm residue symbols (Theorem \ref{thm.product}) and that the intersection number is defined on the level of cycles. 

\begin{proof} i) Since an element of $J_{M,\mca{K}}$ has only finite nontrivial longitudes, the explicit formula of the intersection number on a torus yields that $\sum_K \iota_K$ is a finite sum. 

ii) Let $a,b\in \mca{P}_{M,\mca{K}}$. 
Let $L$ denote the set of $K$'s such that the longitude of $a$ or $b$ is nontrivial and put $V=\cup_{K\subset L} V_K$. 
The boundary map decomposes as $\partial: H_2(M-{\rm Int} V, \partial V) \overset{\partial'}{\To}H_1(\partial V)\to H_1(M-{\rm Int}V)$ and we have $\partial' H_2(M-{\rm Int} V, \partial V) ={\rm Ker}(H_1(\partial V)\to H_1(M-{\rm Int}V))$. 
Noting that the intersection numbers rise to those on cycles, we see that the following diagram commutes. 
\[ \xymatrix{
\partial' H_2(M-{\rm Int}V, \partial V) \times \partial' H_2(M-{\rm Int}V, \partial V) \ar[r] & \Z \ar@{=}[d] \\ 
\partial' H_2(M-{\rm Int}V, \partial V) \times H_2(M-{\rm Int}V,\partial V) \ar@{^{(}-_{>}}[d]  \ar@{->>}[u] \ar[r] & \Z  \ar@{=}[d]\\ 
H_1(\partial V)  \times H_2(M-{\rm Int}V,\partial V) \ar[d] \ar[r] &\Z \ar@{=}[d]\\ 
H_1(M-{\rm Int} V) \times H_2(M-{\rm Int}V, \partial V) \ar[r] & \Z 
} \] 
The first line is the intersection number on $\partial V$. 
The other lines are the maps defined by the intersection numbers on cycles. 
Since the intersection form on the fourth line is non-degenerate, by the commutativity of this diagram and the exactness of the lower three lines, we see that the second line is indeed a zero map. Hence so is the first line and we have $\iota(a,b)=0$. 
\end{proof}

Note that since $\mca{P}_{M,\mca{K}}={\rm Ker}(\mca{J}_{M,\mca{K}}\surj \mca{C}_{M,\mca{K}})$, 
the assertion (2) induces a natural homomorphism $\iota:\mca{C}_{M,\mca{K}}\times \mca{P}_{M,\mca{K}}\to \Z$.

\begin{cor} Suppose in addition that $\mca{K}$ is admissible. 
For any $b \in \mca{P}_{M,\mca{K}}$ and $0\neq n\in \Z$, 
let $h_b$ denote the branched abelian cover corresponding to the kernel of the homomorphism 
$\iota(\bullet, b)\ {\rm mod}\ n: \mca{C}_{M,\mca{K}}\surj \Z/n\Z$. 
In addition, let $S \in H_2(M,\mca{K})$ with $\Delta S=b$ and let $\wt{S}$ denote the surface obtained by capping the meridian disks, that is, let $L$ denote the finite link consisting of all knots on which the longitude of $b$ is nontrivial and let $\wh{S}$ denote the inverse image of $S$ via the natural map $H_2(M,L)\inj H_2(M,\mca{K})$. 
Then, $h_b$ is a $\Z/n\Z$-cover branched along the link $\partial \wh{S}$. 
\end{cor} 

\begin{proof} Let $V=\cup_{K\subset L} V_K$ as before and write $\mca{J}_{M,\mca{K}}\surj H_1(\partial V)\to H_1(M-{\rm Int}V); a\mapsto [a]$. Since the linking number is given by the intersection number, we obtain  
$\ds \iota(a,b)=\iota ([a], S) = {\rm lk}([a], \partial S)$, hence the assertion. 
\end{proof}
Thus, if we regard $\iota_K$ as an analogue of the Hilbert symbol, then an analogue of the Kummer extension is the $\Z/n\Z$-cover branched along a knot. This observation supports the argument of \cite[Appendix]{Hirano1-arXiv}. 

The assumption that $\mca{K}$ is admissible is used only to say the existence of the cover $h_b$. 
In this situation, in the commutative diagram in the proof of Theorem \ref{thm}, the map 
from the third line to the fourth line is surjective. 

We remark that if $M$ is not a $\Q$HS$^3$, then we do not necessarily have a Seifert surface of a knot, and hence we need to choose a longitude for each $K$ in $\mca{K}$. 
The existence of a (mod $n$) Seifert surface of the branch locus seems to correspond to the assumption that ``$k$ contains primitive $n$-th roots of unity'' in Section \ref{sec.Hilbert}.

The study of explicit formulas of the Hilbert symbol has a long history in number theory (cf. Kummer \cite{Kummer1858Hilbert}, Artin--Hasse \cite{ArtinHasse1928Hilbert}, Br\"uckner \cite{Bruckner1967}, Dainobu \cite{Dainobu2022Hilbert}, et.al.). 
%
We wonder if our Hilbert symbol is compatible with that for foliated dynamical systems due to Morishita et.al. in \cite{KimMorishitaNodaTerashima2021} and extends to various explicit formulas. 

\section*{Acknowledgments} 
We are grateful to Masanori Morishita for raising an interesting problem, 
Tomoki Mihara for fruitful discussion, 
the organizers of the conference ``Kyushu Algebraic Number Theory 2021 Spring --hybrid--'' for their great hospitality, 
and the anonymous referees of the journal for essential comments. 
The second author has been partially supported by JSPS KAKENHI Grant Number JP19K14538. 

\bibliographystyle{amsalpha}
\bibliography{ju.Hilbert.arXiv.bbl} 

\end{document}